\documentclass[a4paper,12pt]{article}

\usepackage{titlesec}
\usepackage[utf8]{inputenc}
\usepackage[english]{babel}
\usepackage{hyperref}
 \usepackage{enumerate}
\usepackage{amsmath}
\usepackage{amsthm}
\usepackage{amsfonts}
\usepackage{textcomp} 
 \usepackage{mathtools}
\usepackage{enumitem}
\usepackage[mathscr]{eucal}

\DeclareMathAlphabet{\mathpzc}{OT1}{pzc}{m}{it}
\theoremstyle{plain}
\newtheorem*{theorem*}{Theorem}
\newtheorem{theorem}{Theorem}
\newtheorem{corollary}[theorem]{Corollary}
\newtheorem{proposition}[theorem]{Proposition}

\newtheorem{lemma}[theorem]{Lemma}
\newtheorem{remark}[theorem]{Remark}

\newcounter{para}

\newcommand{\pr}{\mathbb{P}}

\newcommand{\Dcal}{\mathcal{D}}

\newcommand{\R}{\mathbb{R}}

\newcommand{\Z}{\mathbb{Z}}

\renewcommand{\Z}{\mathbf{Z}}
\renewcommand{\Dcal}{\mathcal{D}}

\newcommand{\calP}{\mathcal{P}}

\titleformat{\subsection}[runin]{}{}{}{}[]

\titlespacing{\subsection}{0.0em}{.5ex}{.5ex}[] 

\title{Monotonicity of escape probabilities \\for branching random walks on $\Z^{d}$\vspace{-2mm}}
\author{\textbf{Achillefs Tzioufas}\thanks{\noindent \textsc{Department of Mathematics,   \protect\\  University of Innsbruck,  \protect\\ 
Technikerstrasse 13, 6020,   \protect\\ Innsbruck, Austria. \protect\\}  E-mail: \texttt{achillefs.tzioufas@uibk.ac.at}}}
\begin{document}

\maketitle
\vspace{-5ex}

\begin{abstract} 
\noindent We study nearest-neighbors branching random walks started from a point at the interior of a hypercube. We show that the probability that the process escapes the hypercube is monotonically decreasing with respect to the distance of its starting point from the boundary. We derive as a consequence that at all times the number of particles at a site is monotonically decreasing with respect to its distance from the starting point. 
\end{abstract} 


\section{Statement of Results}
{\it The aperiodic nearest-neighbors branching random walk} with offspring distribution $\calP = (P_{i})_{i \geq 0}$ is a discrete-time particle system  $(\beta_{t})_{t \geq 0}$ on the integer lattice $\Z^{d}$, evolving according to the following rule. Every particle generates independently of other particles a random number of offspring distributed according to $\calP$ and dies after one time unit; each offspring particle moves at a position chosen uniformly at random from the $2d+1$  sites located at most at distance $1$ from the location of the parent. We refer to \cite{LZ} and the references therein for background and recent results on this extensively studied process. 
Hereafter, we let  $\Lambda_{n}^{d} =  \{  x \in \Z^{d} \colon \| x \| \leq n\}$ be the $d$-dimensional hypercube of side length $2n$ centred at the origin, and we let $\partial \Lambda_{n}^{d} = \{  x \in \Lambda_{n}^{d} \colon \| x \| = n\}$ be its boundary, and we let also $\mathring{\Lambda}_{n}^{d} = \Lambda_{n}^{d} \backslash \partial \Lambda_{n}^{d}$ be its interior, where $ \| x \| = \max\{ |x_{i}| \colon i =1, \dots, d\}$. 
We define the usual stochastic domination as follows. If $X$ and $Y$ are random variables with domain $\mathbb{D}$ and cumulative distribution functions, respectively, $F$ and $G$, we write $X \leq_{st.} Y$ to denote 
that: $F(x) \geq G(x)$, for all $x \in \mathbb{D}$.\footnote{where we note that this condition is easily seen to be weaker than, and is due to the so-called quantile coupling in fact equivalent to, that: there exists a joint law  $H$ of $X$ and $Y$ with marginals, respectively, $F$ and $G$, such that $H(X \leq Y) = 1$.}
Further, we define the usual partial order on the positive orthant as follows. If $x,y \in \Z_{+}^{d}$, then we write $x \preceq y$ whenever: $x_{i} \leq y_{i}$, for all $ i= 1, \dots, d$. \footnote{\textup{where we note that, if $x,x'$ are such that $x'$ is obtained by $x$ by reflection over a coordinate axis, then $\tau_{n}^{x} \stackrel{st.}=  \tau_{n}^{x'}$, and hence considering $x,y \in \Z_{+}^{d}$ in Theorems \ref{thmbr} and \ref{thmbr2} and, likewise in Corollary \ref{cor1}, incurs no loss of generality.}} 

\begin{theorem}\label{thmbr} 
We let $(\beta_{t}^{x})$ be the aperiodic nearest-neighbors branching random walk on  $\Z^{d}$ initiated by a particle at $x$. We let $ \tau_{n}^{x} = \inf\{ t\colon \beta_{t}^{x}(y) \geq 1, \mbox{for some } y \in \partial \Lambda_{n}^{d}\}$ and we let $p_{n}(x) = \pr(  \tau_{n}^{x}  < \infty)$, $ x \in \mathring{\Lambda}_{n}^{d}$. We have that 
\begin{equation}\label{thm1des}
x \preceq y  \hspace{2mm}\Longrightarrow\hspace{2mm} \tau_{n}^{y} \leq_{st.}  \tau_{n}^{x}, 
\end{equation}
and, in particular, $p_{n}(x) \leq p_{n}(y)$, for all $n \geq 1$. 
\end{theorem} 


The proof of Theorem \ref{thmbr} comprises a stochastic comparison that relies on spatial-symmetry and uses the assumption that offspring may stay at the same location as their parent  in an essential way. We next consider the corresponding process in which this assumption is dropped.  {\it The periodic  nearest-neighbors branching random walk} with offspring distribution $\calP = (P_{i})_{i \geq 0}$ is a discrete-time particle system  $(\delta_{t})_{t \geq 0}$ on the integer lattice $\Z^{d}$ that evolves according to the following rules. Every particle generates independently of other particles a random number of offspring distributed according to $\calP$ and dies after one time unit. Each offspring particle moves at a position chosen uniformly at random from the $2d$ neighbors of the site of the parent. 

\begin{theorem}\label{thmbr2} 
We let $(\delta_{t}^{x})$ be the periodic nearest-neighbors branching random walk on $\Z^{d}$ initiated by a particle at $x$. We let $ \tau_{n}^{x} = \inf\{ t\colon \delta_{t}^{x}(y) \geq 1, \mbox{for some } y \in \partial \Lambda_{n}^{d}\}$ and let $p_{n}(x) = \pr(  \tau_{n}^{x}  < \infty)$, $ x \in \mathring{\Lambda}_{n}^{d}$. We have that
\begin{equation}\label{thm2des}
x \preceq y  \hspace{2mm}\Longrightarrow\hspace{2mm}  p_{n}(x) \leq p_{n}(y),
\end{equation}
for all  $n \geq 1$. 
\end{theorem} 
Theorem \ref{thmbr2} is shown via a proof by contradiction argument that reduces the proof to two spatial-symmetry stochastic comparisons, a consequence of one of which is given in the next statement.  We next consider the {\it continuous-time nearest-neighbors branching random walk} parameter $\lambda>0$.  This is a particle system $(\zeta_{t})_{t \in \R_{+}}$ on $\Z^{d}$ in which each particle dies after an independent intensity 1 exponential time and generates offspring at each of the $2d$ neighboring sites according to independent intensity $\lambda>0$ inter-arrival exponential times, and, we allow for initial states with only a finite number of particles.  Formally, $\zeta_{t} = (\zeta_{t}(x)\colon x\in \Z^{d})$ is Markov process with state-space $\{\zeta \colon \sum_{x} \zeta(x) < \infty\}$ and transition rates:  $\zeta_{t}(x) \rightarrow \zeta_{t}(x) - 1$, at rate $\zeta_{t}(x)$, and, $\zeta_{t}(x) \rightarrow \zeta_{t}(x) + 1$, at rate $\lambda \sum_{|y- x| = 1} \zeta_{t}(y)$, where $|\cdot|$ denotes $L^{1}$-distance, and where $\sum_{x}\zeta_{0}(x) < \infty$.  We refer to \cite{BZ} and the references therein for background and recent results on this extensively studied process

\begin{corollary}\label{cor1} 
We let $(\zeta_{t}^{o})$  be the continuous-time nearest-neighbors branching random walk on  $\Z^{d}$ initiated by a particle at the origin. We have that
\begin{equation}\label{cor1des}
 x \preceq y  \hspace{2mm}\Longrightarrow\hspace{2mm}  \zeta_{t}^{o}(x) \leq_{st.} \zeta_{t}^{o}(y), 
\end{equation}
for all $t \geq 0$. 
\end{corollary}

 \section{Proofs}
We explicitly treat the planar case only, since the modifications required for higher dimensions are in all cases straightforward. 
 
\begin{proof}[Proof of Theorem \ref{thmbr}.]
To simplify notation, we write $(\beta_{t}^{0})_{t \geq 0}$ for $(\beta_{t}^{(0, 0)})_{t \geq 0}$ and $(\beta_{t}^{1})_{t \geq 0}$ for $(\beta_{t}^{(1, 0)})_{t \geq 0}$ and, further, we write $\tau_{n}^{0}$ for $\tau_{n}^{(0, 0)}$, and $\tau_{n}^{1}$ for $\tau_{n}^{(1, 0)}$.  We will prescribe a stochastic comparison\footnote{i.e., a way to define $(\beta_{t}^{0})_{t \geq 0}$ and $(\beta_{t}^{1})_{t \geq 0}$ on a common probability space with the correct marginal distribution for each one of them.} such that   
\begin{equation}\label{desbrw}
\tau_{n}^{1} \leq \tau_{n}^{0}, \mbox{ a.s..}
\end{equation}
where, we note that (\ref{desbrw}) implies (\ref{thm1des}) and suffices to complete the proof, since the law of $(\beta_{t}^{(1,0)})$ is invariant with respect to permutations of the coordinates of $(1,0)$. The following definitions are in order. We let $V_{r} = \{ (x,y) \in \Z^{2} \colon x = r\}$ where $r \in \Z \cup \{ 1/2\}$. We also let 
$H_{k} = \{ (x,y) \in \Z^{2} \colon y = k\}, k \in \Z$.  We note that $V_{1/2}$ will play the role of a symmetry axis in our stochastic comparison below. We let 
\[
W_{1/2} = V_{0} \cup V_{-1}, \dots \cup V_{-n+1} \mbox{ and } E_{1/2} = V_{1} \cup V_{2}, \dots\cup V_{n}, 
\]
We let $\phi$ be the bijection of every point in $W_{1/2}$ to its mirror image point with respect to $V_{1/2}$ in $E_{1/2}$, which is, we let 
\[
\phi\colon (x,y) \mapsto (-x+1, y), 
\] 
$(x,y) \in W_{1/2}$, where we note that $\phi$ maps points on $V_{-n+1}  \cap \Lambda_{n}^{2}$ to points with the same second coordinate on $V_{n} \cap \Lambda_{n}^{2}$, i.e.,   that
\begin{equation}\label{note0}
\phi(-n+1,y) = (n,y), 
\end{equation}
$y \in \Z$.  In addition, we let 
\[
\Phi(f(x,y)) = f(\phi(x,y)),
\]
$f \colon \Lambda_{n}^{2} \rightarrow \{0,1,\dots\}$. 


The next statement regards a stochastic comparison of $(\beta_{t}^{0})$ and $(\beta_{t}^{1})$ which allows us to decompose each one of them into two non-negative processes which are related amongst them. The idea of the proof of Lemma \ref{deco1} next is a more elaborate version of that of Lemma \ref{gdeco1} below, and thus it is more accessible in the context of the latter statement; see Remark \ref{remgdeco1} below in this regard. 
\begin{lemma}\label{deco1}
We have that 
\begin{equation}\label{eq11}
\beta_{t}^{0} = \sigma_{t}^{0}  + \alpha_{t}^{0} \hspace{2mm}\mbox { and }\hspace{2mm} \beta_{t}^{1} = \sigma_{t}^{1}  + \alpha_{t}^{1}, 
\end{equation}
where $( \sigma_{t}^{0})$, $(\alpha_{t}^{0})$, $( \sigma_{t}^{1})$, $(\alpha_{t}^{1})$  satisfy the following properties:
\begin{equation}\label{eq12}
\sigma_{t}^{0} = \sigma_{t}^{1},  
\end{equation}
\begin{equation}\label{eq14}
\alpha_{t}^{0}(x,y) = 0, \forall (x,y) \in E_{1/2}, \mbox{ and } \alpha_{t}^{1}(x,y) = 0, \forall (x,y) \in W_{1/2},
\end{equation}
\begin{equation}\label{eq13}
\Phi(\alpha_{t}^{0}) = \alpha_{t}^{1}, 
\end{equation}
where we let $\alpha_{0}^{0}(0,0)  = 1$ and, $\alpha_{0}^{0}(x,y)  = 0$, otherwise; and, we let $\alpha_{0}^{1}(1,0) = 1$ and, $\alpha_{0}^{1}(x,y)  = 0$, otherwise; and we also let $\sigma_{0}^{0}(x,y) = \sigma_{0}^{1}(x,y) = 0$, $\forall (x,y)$.
\end{lemma}

We first show that Lemma \ref{deco1} suffices to complete the proof, which is, that (\ref{desbrw}) follows from Lemma \ref{deco1}. We then give the remaining proof of Lemma \ref{deco1}. 

We let $S_{n}^{i} = \inf\{ t\colon \sigma_{t}^{i}(y) \geq 1, \mbox{for some } y \in \partial \Lambda_{n}^{2}\}$, and, furthermore, we let  $U_{n}^{i} =  \inf\{ t\colon \alpha_{t}^{i}(y) \geq 1, \mbox{for some } y \in (H_{n} \cup H_{-n})\cap \Lambda_{n}^{2} \}$, $i =0,1$. 
In addition, we let $T_{n}^{0} =  \inf\{ t\colon \alpha_{t}^{0}(y) \geq 1, \mbox{for some } y \in  V_{-n} \cap \Lambda_{n}^{2} \},$ and, furthermore, we let  $T_{n}^{1} =  \inf\{ t\colon \alpha_{t}^{1}(y) \geq 1, \mbox{for some } y \in  V_{n} \cap \Lambda_{n}^{2} \}$. Note that by (\ref{eq14}), we have that 
\begin{equation}\label{tau20}
\tau_{n}^{i}   = \min(S_{n} ^{i}, T_{n}^{i}, U_{n}^{i}), \mbox{ }  i =0,1. 
\end{equation}
By (\ref{eq12}) we have that $S_{n} ^{0} = S_{n} ^{1}$ and, by (\ref{eq13}), we have that $U_{n}^{0}  = U_{n}^{1}$ and, by (\ref{note0}) that $T_{n}^{1} \leq T_{n}^{0}$. Thus, by (\ref{tau20}), we have that $\tau_{n}^{1} \leq  \tau_{n}^{0}$, a.s., which implies (\ref{desbrw}).

\begin{proof}[Proof of Lemma \ref{deco1}.]
Introducing the following terminology is required in order to give the construction. Particles in $(\sigma_{t}^{0})$ will be called $\sigma^{0}$ particles, particles in $(\sigma_{t}^{1})$ will be called $\sigma^{1}$ particles, and, similarly for particles in $(\alpha_{t}^{0})$ and in $(\alpha_{t}^{1})$. If $\sigma^{0}, \sigma^{1}$ is a pair of particles positioned at the same site, we then say that their offspring is distributed (at the next time instant) in an \textit{identical} way whenever the number of offspring produced by both particles is matched (i.e.\ it is defined by means of the same r.v.) and, furthermore, the displacement of their progeny is determined identically (i.e.\ by means of the same r.v.'s).  Further, if $\alpha^{0}, \alpha^{1}$ is a pair of particles positioned at some site on $(x,y) \in V_{0}$ and on $\phi(x,y) \in V_{1}$ respectively, we then say that offspring of $\alpha^{0}$ and $\alpha^{1}$ is distributed in a \textit{synthetic} way whenever the number of offspring produced by both particles is matched and, furthermore:

\begin{enumerate}[label=(\roman*)]
\item offspring of $\alpha^{0}$ placed at $V_{1}$ are \textbf{relabelled} as $\sigma^{0}$  particles and match offspring of $\alpha^{1}$ placed at $V_{1}$ which is  \textbf{relabelled} as $\sigma^{1}$ particles,
\item offspring of $\alpha^{0}$ placed at $V_{0}$  is  \textbf{relabelled} as $\sigma^{0}$ particles and match offspring of $\alpha^{1}$ placed at $V_{0}$ which is  \textbf{relabelled} as  $\sigma^{1}$ particles,
\item offspring of $\alpha^{0}$ placed at $V_{-1}$ matches offspring of $\alpha^{1}$ placed at $V_{2}$. 
\end{enumerate}
\noindent Further, if $\alpha^{0}$  and $\alpha^{1}$ is a pair of particles at site $(x,y) \in W_{1/2} \backslash V_{0}$  and at site $\phi(x,y) \in E_{1/2} \backslash V_{1}$ respectively, we say that offspring of  $\alpha^{0}$ and  $\alpha^{1}$ are distributed in an \textit{antithetic} way whenever the number of offspring of both particles are matched, and, furthermore:
\begin{enumerate}[label=(\roman**)]
\item offspring of $\alpha^{0}$ placed at the North and South neighbor of $(x,y)$ is matched with offspring of $\alpha^{1}$ placed at the North and South neighbor of $\phi(x,y)$ respectively, whereas, 
\item offspring of $\alpha^{0}$ placed at the site West (resp.\ East) of $(x,y)$ is matched with offspring of  $\alpha^{1}$ placed at the site East  (resp.\ West) of $\phi(x,y)$, 
\end{enumerate}
where offspring of  $\alpha^{0}$  and $\alpha^{1}$ particles will also be $\alpha^{0}$ and  $\alpha^{1}$ particles respectively (i.e.\ no relabelling occurs). 

Note that by the definition of the processes (\ref{eq11}),  (\ref{eq12}),   (\ref{eq14}), (\ref{eq13}) hold for $t=1$.  We will assume that (\ref{eq11}),  (\ref{eq12}),  (\ref{eq14}) and (\ref{eq13}) hold for some $t$ and show that they hold for $t+1$.  The evolution of $\sigma_{t}^{0}$, $\sigma_{t}^{1}$, $\alpha_{t}^{0}$, $\alpha_{t}^{1}$ at time $t+1$ may then be described as follows. 
\begin{enumerate}[label=(\alph*)]
\item  Offspring of $\sigma^{0}$ particles is distributed at time $t+1$ in an identical way to $\sigma^{1}$ particles. 
\item Offspring of $\alpha^{0}$ particles at $(0,y) \in V_{0}$ is distributed in a synthetic way to  $\alpha^{1}$ particles at $\phi(1,y) \in V_{1}$.
\item Offspring of $\alpha^{0}$ particles at $(x,y) \in W_{1/2} \backslash V_{0}$ is distributed in an antithetic way to $\alpha^{1}$ particles at $\phi(x,y) \in E_{1/2} \backslash V_{1}$. 
\end{enumerate}
Note that every $\sigma^{0}$ and $\alpha^{0}$ (resp.\ $\sigma^{1}$ and $\alpha^{1}$) particle generates independently of other $\sigma^{0}$ and $\alpha^{0}$ (resp.\ $\sigma^{1}$ and $\alpha^{1}$) particles a random number of offspring distributed according to $\calP$ and dies after one time unit. Note further that each offspring particle moves at a position chosen uniformly at random from the $5$ sites located at most at distance $1$ from the location of the parent. We thus have that the stochastic comparison described preserves the correct marginal distributions for both $(\beta_{t}^{0})$ and $(\beta_{t}^{1})$, and further, that we have
\begin{equation}\label{eqbetend0a}
\beta_{t+1}^{0} = \sigma_{t+1}^{0}  + \alpha_{t+1}^{0} \hspace{2mm}\mbox { and }\hspace{2mm} \beta_{t+1}^{1} = \sigma_{t+1}^{1}  + \alpha_{t+1}^{1}.  
\end{equation} 

Let $\tilde{\sigma}^{0}_{t+1}$ (resp.\  $\tilde{\sigma}^{1}_{t+1}$) be offspring at time $t+1$ of particles in  $\sigma^{0}_{t}$ (resp.\  $\tilde{\sigma}^{1}_{t}$). Let also $\bar{\sigma}^{0}_{t+1}$ be offspring of $\alpha^{0}$ particles that are located at time $t$ at $V_{0}$ and which are placed at $V_{0}$ and at $V_{1}$ at time $t+1$. Let in addition $\bar{\sigma}^{1}_{t+1}$  be offspring of $\alpha^{1}$ particles that are located at time $t$ at $V_{1}$ and which are placed at $V_{0}$ and at $V_{1}$ at time $t+1$.  Note that 
\begin{equation}\label{sigma01}
\sigma_{t+1}^{i} = \tilde{\sigma}^{i}_{t+1}+  \bar{\sigma}^{i}_{t+1}  \mbox{ for } i = 0,1. 
\end{equation}
By  (\ref{eq12}), we have that  (a) gives that 
\begin{equation}\label{sigmas11}
\tilde{\sigma}^{1}_{t+1} = \tilde{\sigma}^{0}_{t+1}. 
\end{equation}
Note that, by (\ref{eq13}) we have that $\Phi(\alpha_{t}^{0}(0,y))= \alpha_{t}^{1}(1,y)$ and, hence, by  (b), (i), and (ii), we have that 
\begin{equation}\label{sigmas21}
\bar{\sigma}^{1}_{t+1} = \bar{\sigma}^{0}_{t+1}. 
\end{equation}
Plugging (\ref{sigmas11}) and (\ref{sigmas21}) into (\ref{sigma01}) gives that
\begin{equation}\label{sigt0a}
\sigma_{t+1}^{0} = \sigma_{t+1}^{1}.  
\end{equation}  

Note further that from (\ref{eq14}), (b) together with (i) and (ii), give that  
\begin{equation}\label{a00a}
\alpha_{t+1}^{0}(x,y) = 0, \forall (x,y) \in E_{1/2}, \mbox{ and } \alpha_{t+1}^{1}(x,y) = 0,  \forall (x,y) \in W_{1/2}. 
\end{equation}
Note in addition that from (\ref{eq13}), (c) by (i*), (ii*), and (iii), give that
\begin{equation}\label{phi00a}
\Phi(\alpha_{t+1}^{0}(x,y)) = \alpha_{t+1}^{1}(x,y). 
\end{equation}

Hence, by (\ref{eqbetend0a}), (\ref{sigt0a}), (\ref{a00a}) and (\ref{phi00a}) we have that (\ref{eq11}),  (\ref{eq12}),  (\ref{eq14}) and (\ref{eq13}) respectively hold for $t+1$, and the proof is complete. 
\end{proof}
\end{proof}

\begin{proof}[Proof of Theorem \ref{thmbr2}.] We rely on the two following statements, the proofs of which are deferred to immediately below next. Note that  Theorem \ref{thmbr2} for $n=1$ is trivially true and, thus, we may henceforth assume that $n \geq 2$.  
\begin{proposition}\label{prop1}  
We have that, for all  $n \geq 2$,  
\begin{equation}\label{prop1preq2}
\tau_{n}^{(0,0)} \leq_{st.}  \tau_{n}^{(0,2)},
\end{equation}
and, in particular,
\begin{equation}\label{prop1preq12}
p_{n}(0,0) \leq p_{n}(0,2). 
\end{equation}
\end{proposition}

\begin{proposition}\label{prop2}
We have that, for all  $n \geq 2$, 
\begin{equation}\label{preq0}
\tau_{n}^{(0,0)} \leq_{st.}  \tau_{n}^{(1,1)},
\end{equation}
and, in particular, 
\begin{equation}\label{preq}
p_{n}(0,0) \leq p_{n}(1,1). 
\end{equation}
\end{proposition}
We may then derive (\ref{thm2des}) as follows. Let $q_{n}(x) = 1- p_{n}(x)$, $x \in \mathring{\Lambda}_{n}^{2}$.   Observe that by the Markov property, independence of offspring of the particles generated at $t=1$ and, symmetry, we have that 
\begin{equation}\label{abs00}
q_{n}(0,0) = \sum_{i\geq1} q_{n}^{i}(1,0)P_{i}. 
\end{equation} 
Let $X_{1} = \delta_{1}^{(1,0)}(1,1), X_{2}= \delta_{1}^{(1,0)}(2,0), X_{3}= \delta_{1}^{(1,0)}(1,-1), X_{4}=\delta_{1}^{(1,0)}(0,0)$ and let $C = X_{1} + X_{2} +X_{3} + X_{4}$, where $C \sim \calP$, and let $f_{i}( x_{1}, x_{2}, x_{3}, x_{4}) = \pr((X_{1}, X_{2}, X_{3}, X_{4}) =( x_{1}, x_{2}, x_{3}, x_{4})|C=i)$, $i \geq 1$, $ x_{1} +x_{2} + x_{3}+x_{4} = i$.   By the Markov property at time $1$ and symmetry, we have that 
\begin{eqnarray}\label{cconi}
\pr(\tau_{n}(1,0) = \infty | C = i) & = & \sum_{x_{1}+x_{2}+x_{3}+x_{4} = i}f_{i}( x_{1}, x_{2}, x_{3}, x_{4})  q_{n}^{x_{1} +x_{3}}(1,1)q_{n}^{x_{2}}(2,0) q_{n}^{x_{4}}(0,0)\nonumber \\
& \leq & \sum_{x_{1}+x_{2}+x_{3}+x_{4} = i} f_{i}( x_{1}, x_{2}, x_{3}, x_{4}) q_{n}^{i}(0,0)  \nonumber\\
& \leq & q_{n}^{i}(0,0), 
\end{eqnarray}
$i \geq 1$, where in the second line we invoked (\ref{prop1preq12}) and (\ref{preq}). Thus, by (\ref{cconi}), we have that 
\begin{equation}\label{eq10}
q_{n}(1,0) \leq \sum_{i \geq 1}q_{n} ^{i}(0,0) P_{i}. 
\end{equation}
Let us assume to the end of reaching a contradiction that 
\begin{equation}\label{abs1}
q_{n}(0,0) < q_{n}(1,0).  
\end{equation}
Then, by (\ref{abs00}) and (\ref{eq10}), we have that  
\[
\sum_{i\geq1} (q_{n}^{i}(1,0) - q_{n} ^{i}(0,0))P_{i} <  0, 
\]
and we have thus reached a contradiction of (\ref{abs1}) and, hence, we conclude that  $q_{n}(0,0) \geq q_{n}(1,0)$.

\begin{proof}[Proof of Proposition \ref{prop2}.]

Since (\ref{preq}) follows from  (\ref{preq0}) it suffices that we show only the latter. We let $(\beta_{t}^{(0,0)})_{t \geq 0}$ and $(\beta_{t}^{(1,1)})_{t \geq 0}$ be the processes started by a single particle at $(0, 0)$ and at $(1, 1)$ respectively. We show a stochastic comparison among $(\beta_{t}^{(0,0)})_{t \geq 0}$ and $(\beta_{t}^{(1,1)})_{t \geq 0}$ such that  
\begin{equation}\label{desbrw2}
\tau_{n}^{(0,0)} \leq \tau_{n}^{(1,1)}, \mbox{ a.s..} 
\end{equation}
To the end of presenting our stochastic comparison, the following definitions are required. We let $\tilde{\Lambda}^{2}_{n}$ be the box of side length $2n-1$ centred at $(1/2,1/2)$ with its 4 end vertices, $(-n+1, -n+1), (-n+1, n), (n,n), (n,-n+1)$, removed. Further,  we let 
\[
D_{k} = \{ (x,y) \in \Z^{2} \colon y = -x+k\}, k \in \Z, 
\] 
and, we let 
\[
\tilde{D}_{k} = D_{k} \cap \tilde{\Lambda}^{2}_{n}, \mbox{ for } k = -2n+2, -2n+1, \dots, 2n-1, 2n.
\]
where $\tilde{D}_{1}$ will be our symmetry axis in the argument below, since it divides $\tilde{\Lambda}^{2}_{n}$ into two, as follows. Let 
\[
SW_{1/2} = \tilde{D}_{0} \cup \tilde{D}_{-1} \dots \cup \tilde{D}_{-2n+2} \mbox{ and } NE_{1/2} = \tilde{D}_{2} \cup \tilde{D}_{3} \dots \cup \tilde{D}_{2n}. 
\]
We let $\upsilon$ be the bijection of every point in $SW_{1/2}$ to its mirror image point (with respect to $\tilde{D}_{1}$) in $NE_{1/2}$. We in addition let 
\[
\Upsilon(f(x,y)) = f(\upsilon(x,y))
\]
$f \colon \tilde{\Lambda}_{n}^{2} \rightarrow \{0,1,\dots\}$.  We note that $\upsilon$ maps points on $\tilde{D}_{0}, \tilde{D}_{-1}, \dots, \tilde{D}_{-2n+2}$ to points on $\tilde{D}_{2} , \tilde{D}_{3}, \dots, \tilde{D}_{2n}$ respectively and, furthermore, that, if we let $\partial_{i} \tilde{\Lambda}_{n}^{2}, i = N, S, E,W$ be the set of points comprising the North, South, East and West sides of $\partial\tilde{\Lambda}_{n}^{2}$ respectively, we then have that 
\begin{equation}\label{note2}
\upsilon(\partial_{W} \tilde{\Lambda}_{n}^{2}) = \partial_{N} \tilde{\Lambda}_{n}^{2}, \mbox{ } \upsilon(\partial_{S} \tilde{\Lambda}_{n}^{2}) = \partial_{E} \tilde{\Lambda}_{n}^{2}. 
\end{equation}
and, that
\begin{equation}\label{timesint}
\partial_{W}\tilde{\Lambda}_{n}^{2} \in \mathring{\Lambda}_{n}^{2} \mbox{ and } \partial_{S}\tilde{\Lambda}_{n}^{2} \in \mathring{\Lambda}_{n}^{2}, 
\end{equation}
and, that
\begin{equation}\label{timesint2}
\partial_{N}\tilde{\Lambda}_{n}^{2} \subset \partial_{N}\Lambda_{n}^{2} \mbox{ and } \partial_{E}\tilde{\Lambda}_{n}^{2} \subset \partial_{E}\Lambda_{n}^{2}, 
\end{equation}
where $\partial_{i}\Lambda_{n}^{2}$, $i=N,E$, are the set of the points at the North and the East sides of $\Lambda_{n}^{2}$ respectively.

We now give a stochastic comparison between $(\beta_{t}^{(0,0)})$ and $(\beta_{t}^{(1,1)}))$ that allows us to decompose each of them into two processes which are related amongst them. 

\begin{lemma}\label{decodia}
We have that 
\begin{equation}\label{eq011}
\beta_{t}^{(0,0)} = \sigma_{t}^{(0,0)}  + \alpha_{t}^{(0,0)} \hspace{2mm}\mbox { and }\hspace{2mm} \beta_{t}^{(1,1)} = \sigma_{t}^{(1,1)}  + \alpha_{t}^{(1,1)}, 
\end{equation}
where we have that the following properties regarding the auxiliary processes $( \sigma_{t}^{(0,0)})$, $(\alpha_{t}^{(0,0)})$, $( \sigma_{t}^{(1,1)})$, $(\alpha_{t}^{(1,1)})$ hold for all $t\geq0$,  
\begin{equation}\label{eq012}
\sigma_{t}^{(0,0)} = \sigma_{t}^{(1,1)},  
\end{equation}
\begin{equation}\label{eq014}
\alpha_{t}^{(0,0)}(x,y) = 0, \forall (x,y) \in NE_{1/2}, \mbox{ and } \alpha_{t}^{(1,1)}(x,y) = 0, \forall (x,y) \in SW_{1/2},
\end{equation}
\begin{equation}\label{eq013}
\Upsilon(\alpha_{t}^{(0,0)}) = \alpha_{t}^{(1,1)}, 
\end{equation}
where we let $\alpha_{0}^{(0,0)}(0,0)  = 1$ and, $\alpha_{0}^{(0,0)}(x,y)  = 0$, otherwise; and further, we let $\alpha_{0}^{(1,1)}(1,1) = 1$ and, $\alpha_{0}^{(1,1)}(x,y)  = 0$, otherwise; and we also let $\sigma_{0}^{(0,0)}(x,y) = \sigma_{0}^{(1,1)}(x,y) = 0$, $\forall (x,y)$.
\end{lemma} 
We show first that  Lemma \ref{decodia}  suffices to complete the proof, that is that is implies (\ref{desbrw2}), whereas the proof of Lemma \ref{decodia} is given immediately next. 

We let $S_{n}^{(i,i)} = \inf\{ t\colon \sigma_{t}^{(i,i)}(y) \geq 1, \mbox{for some } y \in \partial \Lambda_{n}^{2}\}$, for $i =0,1$. 
We also let $T_{n}^{(1,1)} =  \inf\{ t\colon \alpha_{t}^{(1,1)}(y) \geq 1, \mbox{for some } y \in  \partial_{N}\Lambda_{n}^{2}\cup \partial_{E}\Lambda_{n}^{2}\},$ and, furthermore, we let  $T_{n}^{(0,0)} =  \inf\{ t\colon \alpha_{t}^{(0,0)}(y) \geq 1, \mbox{for some } y \in  \partial_{S}\Lambda_{n}^{2} \cup \partial_{W}\Lambda_{n}^{2}\}$. Note that, due to (\ref{eq014}), we have that 
\begin{equation}\label{taus2}
\tau_{n}^{(i,i)}   = \min(S_{n} ^{(i,i)}, T_{n}^{(i,i)}), \mbox{ }  i =0,1. 
\end{equation}
Note that by (\ref{eq012}) we have that $S_{n} ^{(0,0)} = S_{n} ^{(1,1)}$. Further note that, by (\ref{eq013}), due to (\ref{note2}), (\ref{timesint}), (\ref{timesint2}), we have that $T_{n}^{(1,1)} \leq T_{n}^{(0,0)}$ a.s.. Thus, applying (\ref{taus2}) twice gives (\ref{desbrw2}).  

\begin{proof}[Proof of Lemma \ref{decodia}.]
Introducing some terminology is required in order to give the construction. Particles in $(\sigma_{t}^{(0,0)})$ will be called $\sigma^{(0,0)}$ particles, particles in $(\sigma_{t}^{(1,1)})$ will be called $\sigma^{(1,1)}$ particles, and, particles in $(\alpha_{t}^{(0,0)})$ will be called $\alpha^{(0,0)}$ particles, and particles in $(\alpha_{t}^{(1,1)})$ will be called $\alpha^{(1,1)}$ particles.
Let $\sigma^{(0,0)}, \sigma^{(1,1)}$ be a pair of particles positioned at the same site.  We say that their offspring is distributed in an \textit{identical} way whenever both particles produce the same number of offspring and their displacement is also identical.

Let $\alpha^{(0,0)}, \alpha^{(1,1)}$ be a pair of particles positioned at some site $(x,y)$ on $\tilde{D}_{0}$ and at $(\upsilon(x),\upsilon(y))$ on $\tilde{D}_{2}$ respectively.  We say that their offspring (at the next time instant) is distributed in a \textit{synthetic} way whenever the number of offspring of both particles are matched and, furthermore:

\begin{enumerate}[label=(\roman*)]
\item offspring of $\alpha^{(0,0)}$ placed at its North neighbor, i.e.\ at $(x,y+1)$, are \textbf{relabelled} as $\sigma^{(0,0)}$  particles and are matched to offspring of $\alpha^{(1,1)}$ placed at its West neighbor, i.e.\ at $(\upsilon(x)-1, \upsilon(y))$, which are {\bf relabelled} as $\sigma^{(1,1)}$ particles, 
\item offspring of $\alpha^{(0,0)}$ placed at its East neighbor, i.e.\ at $(x+1,y)$, is  \textbf{relabelled} as $\sigma^{(0,0)}$ particles and are matched to offspring of $\alpha^{(1,1)}$ placed at its South neighbor, i.e. at  $(\upsilon(x), \upsilon(y)-1)$, which are { \bf relabelled} as  $\sigma^{(1,1)}$ particles,
\item offspring of $\alpha^{(0,0)}$ placed at its West neighbor, i.e.\ at $(x-1,y)$, are matched to offspring of $\alpha^{(1,1)}$ placed at its North neighbor, i.e.\ at $(\upsilon(x), \upsilon(y)+1)$, and this offspring are {\bf not relabelled}, 
\item offspring of $\alpha^{(0,0)}$ placed at its South neighbor, i.e.\ at $(x,y-1)$, are matched to offspring of $\alpha^{(1,1)}$ placed at its East neighbor, i.e.\ at $(\upsilon(x), \upsilon(y)+1)$, and this offspring are {\bf not relabelled}. 
\end{enumerate}

Let $\alpha^{(0,0)}$  and $\alpha^{(1,1)}$ be a pair of particles at site $(x,y) \in SW_{1/2} \backslash \tilde{D}_{0}$  and at site $(\upsilon(x),\upsilon(y)) \in NE_{1/2} \backslash \tilde{D}_{2}$ respectively. We say that offspring of  $\alpha^{(0,0)}$ and  $\alpha^{(1,1)}$ are distributed in an \textit{antithetic} way whenever offspring of both particles are matched as in the  synthetic way but without relabelling, viz.\ offspring of  $\alpha^{(0,0)}$  and $\alpha^{(1,1)}$ particles will also be  $\alpha^{(0,0)}$ and $\alpha^{(1,1)}$ particles respectively.

Note that  (\ref{eq011}),  (\ref{eq012}),  (\ref{eq014}) and (\ref{eq013}) hold for $t=1$. We will assume that (\ref{eq011}),  (\ref{eq012}),  (\ref{eq014}) and (\ref{eq013}) hold for some time $t$ and show that they hold for $t+1$.  The stochastic comparison of $\sigma_{t}^{(0,0)}$, $\sigma_{t}^{(1,1)}$, $\alpha_{t}^{(0,0)}$, $\alpha_{t}^{(1,1)}$ at time $t+1$ may then be described as follows. 

\begin{enumerate}[label=(\alph*)]
\item  Offspring of $\sigma^{(0,0)}$ particles is distributed at time $t+1$ in an identical way to $\sigma^{(1,1)}$ particles. 
\item Offspring of $\alpha^{(0,0)}$ particles at $(x,y) \in \tilde{D}_{0}$ is distributed in a synthetic way to  $\alpha^{(1,1)}$ particles at $\upsilon(x,y) \in \tilde{D}_{2}$.
\item Offspring of $\alpha^{(0,0)}$ particles at $(x,y) \in SW_{1/2} \backslash \tilde{D}_{0}$ is distributed in an antithetic way to $\alpha^{(1,1)}$ particles at $\upsilon(x,y) \in NE_{1/2} \backslash \tilde{D}_{2}$. 
\end{enumerate}

Note that every $\sigma^{(0,0)}$ and $\alpha^{(0,0)}$ (resp.\ $\sigma^{1}$ and $\alpha^{1}$) particle at time $t$ generates independently of other $\sigma^{(0,0)}$ and $\alpha^{(0,0)}$ (resp.\ $\sigma^{1}$ and $\alpha^{1}$) particles a random number of offspring distributed according to $\calP$ and dies after one time unit. Furthermore, each offspring particle moves at a position chosen uniformly at random from the $4$ neighbors of the site of the parent.  Therefore, we have that the stochastic comparison described preserves the correct marginal distributions for both $(\beta_{t}^{(0,0)})$ and $(\beta_{t}^{(1,1)})$, and further, that we have
\begin{equation}\label{eqbetend}
\beta_{t+1}^{(0,0)} = \sigma_{t+1}^{(0,0)}  + \alpha_{t+1}^{(0,0)} \hspace{2mm}\mbox { and }\hspace{2mm} \beta_{t+1}^{(1,1)} = \sigma_{t+1}^{(1,1)}  + \alpha_{t+1}^{(1,1)}.  
\end{equation} 

Let $\tilde{\sigma}^{(0,0)}_{t+1}$ (resp.\  $\tilde{\sigma}^{(1,1)}_{t+1}$) be offspring at time $t+1$ of particles in  $\sigma^{(0,0)}_{t}$ (resp.\  $\tilde{\sigma}^{(1,1)}_{t}$). Let also $\bar{\sigma}^{(0,0)}_{t+1}$ be offspring of $\alpha^{(0,0)}$ particles that are located at time $t$ at $ \tilde{D}_{0}$ and which are placed at $ \tilde{D}_{1}$ at time $t+1$. Let in addition $\bar{\sigma}^{(1,1)}_{t+1}$  be offspring of $\alpha^{(1,1)}$ particles that are located at time $t$ at $ \tilde{D}_{2}$ and which are placed at $ \tilde{D}_{1}$ at time $t+1$.  Note that 
\begin{equation}\label{sigma00}
\sigma_{t+1}^{(i,i)} = \tilde{\sigma}^{(i,i)}_{t+1}+  \bar{\sigma}^{(i,i)}_{t+1}  \mbox{ for } i = 0,1. 
\end{equation}
By  (\ref{eq012}), we have that  (a) gives that 
\begin{equation}\label{sigmas10}
\tilde{\sigma}^{(1,1)}_{t+1} = \tilde{\sigma}^{(0,0)}_{t+1}. 
\end{equation}
Note that, by (\ref{eq013}) we have that $\Upsilon(\alpha_{t}^{(0,0)}(x,y))= \alpha_{t}^{(1,1)}(x,y)$,  $(x,y) \in \tilde{D}_{0}$  and, hence, by  (b), (i), and (ii), we have that 
\begin{equation}\label{sigmas20}
\bar{\sigma}^{(1,1)}_{t+1} = \bar{\sigma}^{(0,0)}_{t+1}. 
\end{equation}
Plugging (\ref{sigmas10}) and (\ref{sigmas20}) into (\ref{sigma00}) gives that
\begin{equation}\label{sigt}
\sigma_{t+1}^{(0,0)} = \sigma_{t+1}^{(1,1)}.  
\end{equation}  
Note in addition that from (\ref{eq014}), (b) together with (i) and (ii), give that  
\begin{equation}\label{a0}
\alpha_{t+1}^{(0,0)}(x,y) = 0, \forall (x,y) \in NE_{1/2}, \mbox{ and } \alpha_{t+1}^{(1,1)}(x,y) = 0, \forall (x,y) \in SW_{1/2},
\end{equation}
Note in addition that from (\ref{eq013}), (c) and (iii), give that
\begin{equation}\label{phi0}
\Phi(\alpha_{t+1}^{(0,0)}(x,y)) = \alpha_{t+1}^{(1,1)}(x,y). 
\end{equation}
Hence, from (\ref{eqbetend}), (\ref{sigt}), (\ref{a0}) and (\ref{phi0}) we have that  (\ref{eq011}),  (\ref{eq012}),  (\ref{eq014}) and (\ref{eq013}) respectively hold for $t+1$ and, therefore, that the proof is complete. 
\end{proof}
\end{proof}
\begin{proof}[Proof of Proposition \ref{prop1}.] Since (\ref{prop1preq12}) follows from  (\ref{prop1preq2}), it suffices that we only show the latter.  We will prescribe a stochastic comparison such that   
\begin{equation}\label{desbrw3}
\tau_{n}^{(2,0)} \leq \tau_{n}^{(0,0)}, \mbox{ a.s..}
\end{equation}
The following definitions are in order. We let $V_{k} = \{ (x,y) \in \Z^{2} \colon x = k\}$ and we let $H_{k} = \{ (x,y) \in \Z^{2} \colon y = k\}, k \in \Z$.  We note that $V_{1}$ will play the role of a symmetry axis in our stochastic comparison below. We let 
\[
W_{1} = V_{0} \cup V_{-1}, \dots \cup V_{-n+2} \mbox{ and } E_{1} = V_{2} \cup V_{3}, \dots\cup V_{n}. 
\]
We let $\psi$ be the bijection of every point in $W_{1}$ to its mirror image point with respect to $V_{1}$ in $E_{1}$, which is, we let $\psi\colon (x,y) \mapsto (-x+2, y), \mbox{ }  (x,y) \in W_{1}$,
where we note in particular that $\psi$ maps points on $V_{-n+2}  \cap \Lambda_{n}^{2}$ to points with the same second coordinate on $V_{n} \cap \Lambda_{n}^{2}$, which is that 
\begin{equation}\label{notefin}
\psi(-n+2,y) = (n,y), 
\end{equation}
$y \in \Z$.  In addition, we let 
\[
\Psi(f(x,y)) = f(\psi(x,y)),
\]
$f \colon \Lambda_{n}^{2} \rightarrow \{0,1,\dots\}$. 

In Lemma \ref{gdeco1} stated next we consider $(\gamma_{t}^{x})$, an inhomogeneous branching random walk which is a generalization of  $(\delta_{t}^{x})$. The corresponding extended formulation of Lemma \ref{gdeco1} will be required in the proof of Corollary \ref{cor1}, where we will invoke this statement in its full generality. Lemma \ref{gdeco1}  regards a stochastic comparison of $(\gamma_{t}^{(0,0)})$ and $(\gamma_{t}^{(2,0)})$ that allows us to decompose each one of these processes into two processes which are related amongst them. 

\begin{lemma}\label{gdeco1}
Let $(\gamma_{t}^{x})$ be the generalization of $(\delta_{t}^{x})$ in which each particle may survive in the subsequent generation with probability $\pi \geq 0$, independently of everything else.  (Note that the process  $(\delta_{t}^{x})$ may then be retrieved by setting $\pi =0$). 
We have that 
\begin{equation}\label{geq11}
\gamma_{t}^{(0,0)} = \sigma_{t}^{(0,0)}  + \alpha_{t}^{(0,0)} \hspace{2mm}\mbox { and }\hspace{2mm} \gamma_{t}^{(2,0)} = \sigma_{t}^{(2,0)}  + \alpha_{t}^{(2,0)}, 
\end{equation}
where $( \sigma_{t}^{(0,0)})$, $(\alpha_{t}^{(0,0)})$, $( \sigma_{t}^{(2,0)})$, $(\alpha_{t}^{(2,0)})$ satisfy the following properties:
\begin{equation}\label{geq12}
\sigma_{t}^{(0,0)} = \sigma_{t}^{(2,0)},  
\end{equation}
\begin{equation}\label{geq14}
\alpha_{t}^{(0,0)}(x,y) = 0, \forall (x,y) \in E_{1}, \mbox{ and } \alpha_{t}^{(2,0)}(x,y) = 0, \forall (x,y) \in W_{1},
\end{equation}
\begin{equation}\label{geq13}
\Psi(\alpha_{t}^{(0,0)}) = \alpha_{t}^{(2,0)}, 
\end{equation}
where we let $\alpha_{0}^{(0,0)}(0,0)  = 1$ and, $\alpha_{0}^{(0,0)}(x,y)  = 0$, otherwise; and, we let $\alpha_{0}^{(2,0)}(2,0) = 1$ and, $\alpha_{0}^{(2,0)}(x,y)  = 0$, otherwise; and we also let $\sigma_{0}^{(0,0)}(x,y) = \sigma_{0}^{(2,0)}(x,y) = 0$, $\forall (x,y)$.
\end{lemma}

\begin{remark}\label{remgdeco1}
\textup{Roughly speaking, the idea behind the proof of Lemma \ref{gdeco1} is as follows. Particles in $\gamma_{t}^{(0,0)}$ belong in $\alpha_{t}^{(0,0)}$ for as long as they occupy sites in $W_{1}$, whereas
those in $\gamma_{t}^{(2,0)}$ belong in $\alpha_{t}^{(2,0)}$ for as long as they occupy sites in $E_{1}$. Whenever particles in $\alpha_{t}^{(0,0)}$ (resp.\ $\alpha_{t}^{(2,0)}$) give birth on $V_{1}$, the newly born particles belong for all future times in $\sigma_{t}^{(0,0)}$ (resp.\ $\sigma_{t}^{(2,0)}$) instead. Since the displacement of descendant particles in $\alpha_{t}^{(0,0)}$ and in $\alpha_{t}^{(2,0)}$ is antithetic w.r.t.\ $V_{1}$, all births on $V_{1}$ occur at the same time in both $\gamma_{t}^{(0,0)}$ and $\gamma_{t}^{(2,0)}$. Since in addition the displacement of descendant particles in $\sigma_{t}^{(0,0)}$ and in $\sigma_{t}^{(2,0)}$ is identical, particles in $\gamma_{t}^{(0,0)}$ (resp.\ $\gamma_{t}^{(2,0)}$) can escape to $E_{1}$ (resp.\ $W_{1}$) only if belonging to $\sigma_{t}^{(0,0)}$ (resp.\ $\sigma_{t}^{(2,0)}$).}

\end{remark}
We show first that  Lemma \ref{gdeco1}  suffices to complete the proof, that is that is implies (\ref{desbrw3}), whereas the proof of Lemma \ref{gdeco1} is given immediately next.

We let $S_{n}^{i} = \inf\{ t\colon \sigma_{t}^{i}(y) \geq 1, \mbox{for some } y \in \partial \Lambda_{n}^{2}\}$,  for $i =(0,0), (0,2)$.  Furthermore, we let  $U_{n}^{i} =  \inf\{ t\colon \alpha_{t}^{i}(y) \geq 1, \mbox{for some } y \in (H_{n} \cup H_{-n})\cap \partial \Lambda_{n}^{2} \}$, for $i =(0,0), (0,2)$. We also let $T_{n}^{(0,0)} =  \inf\{ t\colon \alpha_{t}^{(0,0)}(y) \geq 1, \mbox{for some } y \in  V_{-n} \cap \partial \Lambda_{n}^{2} \},$ and, further, we let  $T_{n}^{(2,0)} =  \inf\{ t\colon \alpha_{t}^{(2,0)}(y) \geq 1, \mbox{for some } y \in  V_{n} \cap \partial \Lambda_{n}^{2} \}$. Note that by (\ref{geq14}), we have that 
\begin{equation}\label{tau2}
\tau_{n}^{i}   = \min(S_{n} ^{i}, T_{n}^{i}, U_{n}^{i}), \mbox{ }  i =(0,0), (0,2). 
\end{equation}
Note that by (\ref{geq12}) we have that $S_{n} ^{(0,0)} = S_{n} ^{(2,0)}$. Further, note that by (\ref{geq13}) we also have that $U_{n}^{(0,0)} = U_{n}^{(2,0)}$. Note in addition that by (\ref{notefin}) and (\ref{geq13}), we also have that $T_{n}^{(2,0)} \leq T_{n}^{(0,0)}$ a.s.. Thus, applying (\ref{tau2}) twice gives (\ref{desbrw3}).

\begin{proof}[Proof of Lemma \ref{gdeco1}.] Introducing the following terminology is required in order to give the construction. Particles in $(\sigma_{t}^{(0,0)})$ will be called $\sigma^{(0,0)}$ particles, particles in $(\sigma_{t}^{(2,0)})$ will be called $\sigma^{(2,0)}$ particles and similarly for $(\alpha_{t}^{(0,0)})$ and $(\alpha_{t}^{(2,0)})$. If $\sigma^{(0,0)}, \sigma^{(2,0)}$ is a pair of particles positioned at the same site, we then say that their offspring is distributed (at the next time instant) in an \textit{identical} way whenever the number of offspring produced by both particles is matched (i.e.\ it is defined by means of the same r.v.) and, furthermore, the displacement of their progeny is determined identically (i.e.\ by means of the same r.v.'s), and in addition, particle $\sigma^{(0,0)}$ survives (at time $t+1$) if and only if $\sigma^{(2,0)}$ survives.  
Further, if $\alpha^{(0,0)}, \alpha^{(2,0)}$ is a pair of particles positioned at some site on $(x,y) \in V_{0}$ and on $\psi(x,y) \in V_{2}$ respectively, we then say that offspring of $\alpha^{(0,0)}$ and of $\alpha^{(2,0)}$ is distributed in a \textit{synthetic} way whenever the number of offspring produced by both particles is matched and, furthermore:
\begin{enumerate}[label=(\roman*)]   
\item offspring of $\alpha^{(0,0)}$ placed at $V_{1}$ at the site East of $(x,y)$  is \textbf{relabelled} as $\sigma^{(0,0)}$  particles and match offspring of  $\alpha^{(2,0)}$ placed at the site West of $\psi(x,y)$ which is {\bf relabelled } as $\sigma^{(2,0)}$ particles,
\item offspring of $\alpha^{(0,0)}$ placed at the site West of $(x,y)$ is matched with  offspring of  $\alpha^{(2,0)}$ placed at the site East of $\psi(x,y)$, 
\item particle $\alpha^{(0,0)}$ survives (at the next time instant) if and only if $\alpha^{(2,0)}$ survives, 
\end{enumerate}

Further, if $\alpha^{(0,0)}$  and $\alpha^{(2,0)}$ is a pair of particles at site $(x,y) \in W_{1} \backslash V_{0}$  and at site $\psi(x,y) \in E_{1} \backslash V_{2}$ respectively, we say that offspring of  $\alpha^{(0,0)}$ and  $\alpha^{(2,0)}$ is distributed in an \textit{antithetic} way whenever the number of offspring of both particles is matched, and, furthermore:
\begin{enumerate}[label=(\roman**)]
\item offspring of $\alpha^{(0,0)}$ placed at the North and South neighbor of $(x,y)$ is matched with offspring of $\alpha^{(2,0)}$ placed at the North and South neighbor of $\psi(x,y)$ respectively, whereas, 
\item offspring of $\alpha^{(0,0)}$ placed at the site West (resp.\ East) of $(x,y)$ is matched with  offspring of  $\alpha^{(2,0)}$ placed at the site East  (resp.\ West) of $\psi(x,y)$, 
\item  particle $\alpha^{(0,0)}$ survives (at the next time instant) if and only if $\alpha^{(2,0)}$ survives, 
\end{enumerate}
where offspring of  $\alpha^{(0,0)}$  and $\alpha^{(2,0)}$ particles will also be $\alpha^{(0,0)}$ and  $\alpha^{(2,0)}$ particles respectively (i.e.\ no relabelling occurs). 

Note that (\ref{geq11}),  (\ref{geq12}),   (\ref{geq14}), (\ref{geq13}) hold for $t=1$.  We will assume that (\ref{geq11}),  (\ref{geq12}),  (\ref{geq14}) and (\ref{geq13}) hold for some $t$ and show that they hold for $t+1$.  The evolution of $\sigma_{t}^{(0,0)}$, $\sigma_{t}^{(2,0)}$, $\alpha_{t}^{(0,0)}$, $\alpha_{t}^{(2,0)}$ at time $t+1$ may then be described as follows. 
\begin{enumerate}[label=(\alph*)]
\item  Offspring of $\sigma^{(0,0)}$ particles is distributed in an identical way to $\sigma^{(2,0)}$ particles. 
\item Offspring of $\alpha^{(0,0)}$ particles at $(0,y) \in V_{0}$ is distributed in a synthetic way to  $\alpha^{(2,0)}$ particles at $\psi(2,y) \in V_{2}$. 
\item Offspring of $\alpha^{(0,0)}$ particles at $(x,y) \in W_{1} \backslash V_{0}  $ is distributed in an antithetic way to $\alpha^{(2,0)}$ particles at $\psi(x,y) \in E_{1} \backslash V_{2}$. 
\end{enumerate} 
Note that every $\sigma^{(0,0)}$ and $\alpha^{(0,0)}$ (resp.\ $\sigma^{(2,0)}$ and $\alpha^{(2,0)}$) particle generates independently of other $\sigma^{(0,0)}$ and $\alpha^{(0,0)}$ (resp.\ $\sigma^{(2,0)}$ and $\alpha^{(2,0)}$) particles a random number of offspring distributed according to $\calP$ and dies after one time unit with probability $\pi$. Note further that each offspring particle moves at a position chosen uniformly at random from the $4$ sites located at most at distance $1$ from the location of the parent. We thus have that the stochastic comparison described preserves the correct marginal distributions for both $(\gamma_{t}^{(0,0)})$ and $(\gamma_{t}^{(2,0)})$, and further, that 
\begin{equation}\label{eqbetend0}
\gamma_{t+1}^{(0,0)} = \sigma_{t+1}^{(0,0)}  + \alpha_{t+1}^{(0,0)} \hspace{2mm}\mbox { and }\hspace{2mm} \gamma_{t+1}^{(2,0)} = \sigma_{t+1}^{(2,0)}  + \alpha_{t+1}^{(2,0)}.  
\end{equation} 

Let $\tilde{\sigma}^{(0,0)}_{t+1}$ (resp.\  $\tilde{\sigma}^{(2,0)}_{t+1}$) be offspring at time $t+1$ of particles in  $\sigma^{(0,0)}_{t}$ (resp.\  $\tilde{\sigma}^{(2,0)}_{t}$). Let also $\bar{\sigma}^{(0,0)}_{t+1}$ be offspring of $\alpha^{(0,0)}$ particles that are located at time $t$ at $V_{0}$ and which are placed at $V_{1}$ at time $t+1$. Let in addition $\bar{\sigma}^{(2,0)}_{t+1}$  be offspring of $\alpha^{(2,0)}$ particles that are located at time $t$ at $V_{2}$ and which are placed at $V_{1}$ at time $t+1$.  Note that 
\begin{equation}\label{sigma0}
\sigma_{t+1}^{i} = \tilde{\sigma}^{i}_{t+1}+  \bar{\sigma}^{i}_{t+1}  \mbox{ for } i = (0,0), (2,0). 
\end{equation}
By  (\ref{geq12}), we have that  (a) gives that 
\begin{equation}\label{sigmas1}
\tilde{\sigma}^{(2,0)}_{t+1} = \tilde{\sigma}^{(0,0)}_{t+1}. 
\end{equation}
Note that, by (\ref{geq13}) we have that $\Psi(\alpha_{t}^{(0,0)}(0,y))= \alpha_{t}^{(2,0)}(2,y)$ and, hence, by  (b), (i), and (ii), we have that 
\begin{equation}\label{sigmas2}
\bar{\sigma}^{(2,0)}_{t+1} = \bar{\sigma}^{(0,0)}_{t+1}. 
\end{equation}
Plugging (\ref{sigmas1}) and (\ref{sigmas2}) into (\ref{sigma0}) gives that
\begin{equation}\label{sigt0}
\sigma_{t+1}^{(0,0)} = \sigma_{t+1}^{(2,0)}.  
\end{equation}  

Note further that from (\ref{geq14}), (b) together with (i), (ii) and (iii), give that  
\begin{equation}\label{a00}
\alpha_{t+1}^{(0,0)}(x,y) = 0, \forall (x,y) \in E_{1}, \mbox{ and } \alpha_{t+1}^{(2,0)}(x,y) = 0,  \forall (x,y) \in W_{1}. 
\end{equation}
Note in addition that from (\ref{geq13}), (c) by (i*), (ii*), (iii*), and (iii), give that
\begin{equation}\label{phi00}
\Psi(\alpha_{t+1}^{(0,0)}(x,y)) = \alpha_{t+1}^{(2,0)}(x,y). 
\end{equation}

Hence, by (\ref{eqbetend0}), (\ref{sigt0}), (\ref{a00}) and (\ref{phi00}) we have that (\ref{geq11}),  (\ref{geq12}),  (\ref{geq14}) and (\ref{geq13}) respectively hold for $t+1$ and, thus, the proof is complete. 
\end{proof}

\end{proof}
\end{proof}

\begin{proof}[Proof of Corollary \ref{cor1}.] 
Let $(\gamma^{(0,0)}_{n}, \gamma^{(2,0)}_{n})_{n \geq 0}$ be as in Lemma \ref{gdeco1}. Let $\Dcal_{2\Z} = \bigcup_{k \in 2\Z} D_{k}$, where $D_{k}= \{(x,y) \in \Z_{+}^{2} \colon y = - x+ k\}$. We first show that  
\begin{equation}\label{eqcorpr}
\gamma_{2n}^{(0,0)}(x,y) \leq_{st.} \gamma_{2n}^{(0,0)}(x+2,y), 
\end{equation}
$\forall (x,y) \in \Dcal_{2\Z}$.  To this end, let $\Dcal_{2\Z}'$ be the set of $(x,y) \in \Dcal_{2\Z}$  such that $x \geq 2$. By  (\ref{geq11}) due to  (\ref{geq14}), we have that if $(x,y) \in \Dcal_{2\Z}'$, then
\begin{eqnarray*}\label{gammasig}
\gamma_{2n}^{(0,0)}(x,y) & =& \sigma_{2n}^{(0,0)}(x,y)  \nonumber \\
 & =& \sigma_{2n}^{(2,0)}(x,y) \leq  \gamma_{2n}^{(2,0)}(x,y), 
\end{eqnarray*}
a.s., where in the second line we used (\ref{geq12}), and then (\ref{geq11}). By invariance of the law of  $(\gamma_{n})$ under shifts of the starting point and reflections over the coordinate axis, the last display above gives that $\gamma_{2n}^{(x,y)}(0,0) \leq_{st.} \gamma_{2n}^{(x,y)}(2,0)$ and hence, that $\gamma_{2n}^{(0,0)}(x,y) \leq_{st.} \gamma_{2n}^{(0,0)}(x-2,y)$, which gives (\ref{eqcorpr}). 

Let $(x,y) \in \Z_{+}^{2}$ and $t \in \R_{+}$ be fixed. Let also $N \geq 1$ be sufficiently large in order that $\sqrt{1-\lambda/N}<1$. Let $^{N}\gamma_{n}^{(0,0)}$ be the process with $\pi_{N} = \sqrt{1-1/N}$ and offspring distribution $\mathcal{P}_{N} = (P_{i, N})_{i \geq 0}$ such that $P_{i,N} =  \pr(\sum_{j =1}^{4} X_{j, N} = i)$, where $X_{j, N}$ are independent Bernoulli parameter $\sqrt{\lambda/N}$.  Based on the fact that, if $B_{n}$ is Binomial parameters $n \geq 1$ and $c/n$, and $B$ is Poisson parameter $c >0$, then, as $n \rightarrow \infty$, $B_{n} \stackrel{d}{\rightarrow} B$, standard arguments (see, for instance, \cite{D}), yield that, as $N \rightarrow \infty$, 
\begin{equation}\label{eqcorfin}
^{N}\gamma_{2 [Nt]}^{(0,0)} (2x, 2y) \stackrel{d}{\rightarrow} \zeta_{t}^{o}(x, y) \hspace{2mm} \mbox{ and } \hspace{2mm} ^{N}\gamma_{2 [Nt]}^{(0,0)} (2(x+1), 2y) \stackrel{d}{\rightarrow} \zeta_{t}^{o}(x+1, y), 
\end{equation}
where $\stackrel{d}{\rightarrow}$, and $[\cdot]$, denote convergence in distribution, and integer part, respectively. Since $\{ (x,y) \in \Z_{+}^{2} \colon x ,y \in 2\Z\} \subset \Dcal_{2\Z}$,    we have that (\ref{eqcorpr}) and  (\ref{eqcorfin}) imply (\ref{cor1des}) and, thus, the proof is complete. 
\end{proof}
\begin{remark}
\textup{We finally note that directly analogously to the proof of (\ref{eqcorpr}) above we may additionally deduce that $\gamma_{2n+1}^{(0,0)}(x,y) \leq_{st.} \gamma_{2n+1}^{(0,0)}(x+2,y)$, $\forall  (x,y) \in \Dcal_{2\Z+1}$,  $\Dcal_{2\Z+1} = \bigcup_{k \in 2\Z+1} D_{k}$, and therefore, in particular that, if $(\delta_{t}^{o})$ is the process considered in Theorem \ref{thmbr2} started by a particle at the origin, then 
\[
 x \preceq y  \hspace{2mm}\Longrightarrow\hspace{2mm}  \delta_{t}^{o}(x) \leq_{st.} \delta_{t}^{o}(y), 
\]
for all $x, y$ such that $\delta_{t}^{o}(x), \delta_{t}^{o}(y) \not= 0$ a.s., $t \geq 0$. Likewise, by invoking Lemma \ref{deco1} instead of Lemma \ref{gdeco1}, we may also derive that if $(\beta_{t}^{o})$ is the process considered in Theorem \ref{thmbr} started by a particle at the origin, then   $x \preceq y  \Longrightarrow \beta_{t}^{o}(x) \leq_{st.} \beta_{t}^{o}(y)$,  result shown by means of a different approach which extends the reflection principle for random walks in \cite{LZ}, Lemma 11.}
\end{remark}



\vspace{3mm}
\noindent \scriptsize{{\bf Acknowledgement}: This work was funded by the Austrian Science Fund (FWF), grant  \# Y760.}
\vspace{-4mm}

\bibliographystyle{plain}

\begin{thebibliography}{1}



\bibitem{BZ} {\sc Bertacchi, D.} and {\sc Zucca, F.}  (2015). Branching random walks and multi-type contact-processes on the percolation cluster of $Z^d$. {\it  The Annals of Applied Probability}, 1993-2012.

\bibitem{D} {\sc  Durrett, R. } (1988). { Lecture notes on particle systems and percolation.} Brooks/Cole Pub Co.


\bibitem{LZ}{\sc Lalley, S. P.} and {\sc  Zheng, X. } (2011). Occupation statistics of critical branching random walks in two or higher dimensions.  { \it The Annals of Probability,}  {\bf 39(1)},  327-368. 






\end{thebibliography}

\end{document}